\documentclass[12pt]{amsart}
\usepackage{amssymb}

\numberwithin{equation}{section}

\begin{document}

{\theoremstyle{theorem}
    \newtheorem{theorem}{\bf Theorem}[section]
    \newtheorem{proposition}[theorem]{\bf Proposition}
    \newtheorem{claim}{\bf Claim}[theorem]
    \newtheorem{lemma}[theorem]{\bf Lemma}
    \newtheorem{corollary}[theorem]{\bf Corollary}
    \newtheorem{acknowledgement}{\bf Acknowledgement}
    \renewcommand{\theacknowledgement}{}
}
{\theoremstyle{remark}
    \newtheorem{remark}[theorem]{\bf Remark}
    \newtheorem{example}[theorem]{\bf Example}
}
{\theoremstyle{definition}
    \newtheorem{definition}[theorem]{\bf Definition}
    \newtheorem{conjecture}[theorem]{\bf Conjecture}
    \newtheorem{question}[theorem]{\bf Question}
    \newtheorem{notation}[theorem]{\bf Notation}
}

\def\C{{\mathcal C}}
\def\H{{\mathcal H}}
\def\NN{{\mathbb N}}
\def\ZZ{{\mathbb Z}}
\def\RR{{\mathbb R}}
\def\a{{\bf a}}
\def\x{{\bf x}}
\def\y{{\bf y}}
\def\z{{\bf z}}
\def\w{{\bf w}}
\def\u{{\bf u}}
\def\v{{\bf v}}
\def\0{{\bf 0}}
\def\1{{\bf 1}}
\def\m{{\mathfrak m}}
\def\height{\operatorname{ht}}

\def\demo{\noindent{\bf Proof. }}
\def\QED{\hfill$\Box$}
\def\qed{\QED}
\def\Ass{\operatorname{Ass}}
\def\Min{\operatorname{Min}}
\def\depth{\operatorname{depth}}
\def\height{\operatorname{height}}
\def\grade{\operatorname{grade}}
\def\mgrade{\operatorname{m-grade}}
\def\rar{\rightarrow}
\def\gin{\operatorname{gin}}

\title{Embedded Associated Primes of Powers of Square-free Monomial Ideals}
\thanks{The first author is partially supported by BOR Grant LEQSF(2007-10)-RD-A-30 and Tulane Research Enhancement Fund.}
 
\author{Huy T\`{a}i H\`{a}}
\address{Tulane University\\
Department of Mathematics\\
6823 St. Charles Avenue\\
New Orleans, LA 70118 }
\email{tai@math.tulane.edu}
\urladdr{http://www.math.tulane.edu/~tai/}

\author{Susan Morey}
\address{Department of Mathematics \\
Texas State University\\
601 University Drive\\ 
San Marcos, TX 78666}
\email{morey@txstate.edu}
\urladdr{http://www.txstate.edu/~sm26/}


\begin{abstract}
An ideal $I$ in a Noetherian ring $R$ is normally torsion-free if $\Ass (R/I^t) = \Ass (R/I)$ for all $t \ge 1$. We develop a technique to inductively study normally torsion-free square-free monomial ideals. In particular, we show that if a square-free monomial ideal $I$ is minimally not normally torsion-free then the least power $t$ such that $I^t$ has embedded primes is bigger than $\beta_1$, where $\beta_1$ is the monomial grade of $I$, which is equal to the matching number of the hypergraph $\H(I)$ associated to $I$. 
If in addition $I$ fails to have the packing property, then embedded primes of $I^t$ do occur when $t = \beta_1 +1$. As an application, we investigate how these results relate to a conjecture of Conforti and Cornu\'ejols. 
\end{abstract}

\maketitle


\section{Introduction}

There is a combinatorial realization of a square-free monomial ideal that can be manifested in a variety of ways, depending on the reader's background. For this paper we will primarily use the language of hypergraph theory. We will view square-free monomial ideals as edge ideals of hypergraphs. A {\it hypergraph} $\H$  consists of a finite set of points $V(\H) = \{x_1, \dots, x_d\}$, called {\it vertices}, and a family $E(\H)$ of nonempty subsets of $V(\H)$, called {\it edges}. A graph in the classical sense is a hypergraph in which every edge has cardinality at most two (we consider an isolated vertex as an edge consisting of one vertex). A hypergraph is {\it simple} if there are no non-trivial containments among the edges (i.e., if $E_1$ and $E_2$ are distinct elements in $E(\H)$ then $E_1 \not\subseteq E_2$). For the purposes of this paper, all hypergraphs will be assumed to be simple, and simple hypergraphs are allowed to have isolated vertices. Simple hypergraphs are also known as {\it Sperner families} or {\it clutters}. 

Let $k$ be a field. By identifying the points in $V(\H)$ with the variables in a polynomial ring $R=k[x_1, \dots, x_d]$, the natural one-to-one correspondence between square-free monomial ideals in $R$ and simple hypergraphs on $\{x_1, \dots, x_d\}$ is given by
$$\H \longleftrightarrow I(\H) = \Big\langle x^E = \prod_{x \in E}x ~\Big|~ E \in E(\H) \Big\rangle.$$
The ideal $I(\H)$ is referred to as the {\it edge ideal} of $\H$. The construction of edge ideals of hypergraphs is an extension of edge ideals of graphs (cf. \cite{FH, HVT1, V1}), and has appeared in the literature (cf. \cite{HVT, Trung, HHTZ}). It is the same as the construction of edge ideals of clutters (cf. \cite{HMV}) and similar to the construction of facet ideals of simplicial complexes (cf. \cite{Fa, HVT1}). We shall also use $\H(I)$ to denote the corresponding hypergraph of a square-free monomial ideal $I$.

For any ideal $I$ in a Noetherian ring $R$, Brodmann \cite{brod} showed that the sets of associated primes of powers of $I$, $\Ass (R/I^t)$, stabilize for large $t$. However, little is known about where the stability occurs (for which $t$), or which embedded primes are in the stable set. When $I$ is a square-free monomial ideal, $\Ass (R/I)=\Min (R/I)$, but powers of $I$ could have embedded primes. In this paper we seek to describe those primes and the powers for which they occur. Since associated primes localize, we will focus on determining when the maximal homogeneous ideal occurs as an associated prime of $I^t$ for some $t \ge 1$.

It was shown in \cite{SVV} that a graph $G$ is bipartite if and only if its edge ideal $I(G)$ is {\it normally torsion-free}, or equivalently, if $I(G)^t$ has no embedded primes for any $t \ge 1$. For non-bipartite graphs, a method of describing embedded associated primes based on odd cycles, and a bound on where the stability occurs, were given in \cite{AJ}. In this paper, we study more general square-free monomial ideals, whose generators are not restricted to having degree at most two. Our focus is to investigate embedded associated primes of powers $I^t$ of a square-free monomial ideal $I$ in the case when every proper {\it minor} (see Section \ref{background} for definitions) of $I$ is normally torsion-free. 

In Theorem \ref{sequences} we relate the associated primes of $I^t$ to those of the colon ideal $(I^t : M)$, where $M$ is a product of distinct variables in $R$. In particular, we show that the maximal homogeneous ideal is associated to one of these ideals if and only if it is associated to the other. An immediate application of this result is given in Corollary \ref{notbeforeB}, where we give a sharp lower bound on the power $t$ for which $I^t$ has embedded primes when $I$ is minimally not normally torsion-free (i.e., $I$ is not normally torsion-free but all its proper minors are). Note that such ideals correspond to a well-studied but little understood class of hypergraphs (cf. \cite{Cornuejols, Schrijver} and their references). We show that in this case, $I^t$ has no embedded primes for all $t \le \beta_1$, where $\beta_1$ is the {\it monomial grade} of $I$; that is, the maximum length of a regular sequence of monomials in $I$. If, in addition, $I$ fails the {\it packing property} (see Section \ref{background} for definitions) then we show in Theorem \ref{thm.embeddedpower} that $I^{\beta_1+1}$ must have an embedded prime. As a consequence of our work, a minimal counterexample to the well-known Conforti-Cornu\'ejols conjecture (cf. \cite{CC} or \cite[Conjecture 1.6]{Cornuejols}) cannot be unmixed, and if the conjecture holds then in Remark \ref{rmk.highestpower} we give a sharp bound on the number of powers $N$ for which one must check the equality $I^{(n)} = I^n$ for $n \le N$ to guarantee that $I$ is normally torsion-free. See \cite[Corollary 3.14]{GVV} and \cite[Corollary 1.6]{HHTZ} for an explanation of how the normally torsion-free condition relates to the combinatorial conditions in the Conforti-Cornu\'ejols conjecture, and see \cite[Conjecture 4.18]{GRV}) for an algebraic translation of the conjecture. Note that due to our focus on when the maximal homogeneous ideal is an associated prime of $I^t$, the further translation given in \cite[Conjecture 4.21]{GRV} is of particular interest to us since $\depth (R/I^t) \geq 1$ precisely when the maximal homogeneous ideal is not in $\Ass(R/I^t)$.

The techniques used in the paper differ by section. In Section \ref{sequencemethods} we employ mainly algebraic techniques. We prove Theorem \ref{sequences} by using a series of short exact sequences to control the associated primes. In Proposition \ref{powersreduce} we use symbolic powers to relate the colon ideal $(I^t:M)$ to smaller powers $I^s$, for $s < t$. By combining Theorem \ref{sequences} and Proposition \ref{powersreduce}, we are able to employ inductive arguments to gain control over the embedded primes of $I^t$. 

In Section \ref{polarization}, we use a more combinatorial approach. Our method in proving Theorem \ref{thm.embeddedpower} is to use polarization. In particular, we use a result of Faridi (see Lemma \ref{primesPolarize}) which gives a correspondence (which is not necessarily one-to-one) between the associated primes of $I^t$ and the associated primes of its polarization. In Theorem \ref{thm.embeddedpower} we use this correspondence to force the maximal homogeneous ideal to be associated to $I^{\beta_1 +1}$ when proper minors of $I$ are normally torsion-free, but $I$ fails to satisfy the packing property.


\section{Preliminaries}\label{background}

In this section we collect notation and terminology that will be used throughout the paper. Note that we will consistently use the term hypergraph to mean a simple hypergraph.

Our primary focus will be the minimal primes,
associated primes and symbolic powers of ideals. A prime $P$ is
{\it minimal} over an ideal $I$ if $I\subseteq  P$ and there does
not exist a prime $Q \not= P$ with $I \subseteq Q \subsetneq P$. The set of minimal primes over $I$ is denoted by $\Min (R/I)$. Since $I$ is a
monomial ideal, all minimal primes are monomial prime ideals, that is, they are
generated by subsets of the variables. 
A prime $P$ is an {\it associated prime} of $I$ if there exists an element $c$ in $R$
such that $P=(I:c)$. Note that all minimal primes are also associated
primes. We say that an ideal $I$ is {\it unmixed} if all of its associated primes have the same height. 

An ideal $I$ has a primary decomposition 
$$I=q_1 \cap \ldots \cap q_t \cap Q_1 \cap \ldots \cap Q_s$$
were $q_i$ and $Q_j$ are primary ideals with $\sqrt{q_i}$ the minimal
primes of $I$. The primes $\sqrt{Q_j}$ are the {\it embedded}
associated primes of $I$. The set of associated primes of an ideal $I \subseteq R$ is denoted by $\Ass(R/I)$. An ideal $I$ is said to be {\it normally torsion-free} if $\Ass(R/I^t)=\Ass(R/I)$ for all $t \geq 1$.

\begin{definition} The $t$-th {\it symbolic power} of an ideal $I$,
denoted by $I^{(t)}$, is the intersection of the primary components of
$I^t$ that correspond to minimal primes of $I$. 
\end{definition}

Notice that if $I$ is a square-free monomial ideal then $\Ass(R/I) = \Min(R/I)$; that is, $I$ has no embedded primes. Thus, in this case, $I$ is normally torsion-free if and only if $I^t=I^{(t)}$ for all $t \ge 1$.  

An important fact that we shall use is that localization
preserves associated primes. That is, if $P$ is a prime ideal
containing an ideal $I$ then $P \in \Ass (R/I^t)$ if and only if $PR_P \in \Ass
(R_P/(IR_P)^t)$. This allows us to reduce our problem to investigating when the maximal ideal $\m = (x_1, \dots, x_d)$ is an embedded associated prime of $I^t$.  

Although the focus of the paper is square-free monomial ideals, due to the correspondence between such ideals and hypergraphs, there are some invariants and terminology that come from discrete mathematics that will be useful. See, for example, \cite{Berge} or \cite{Cornuejols} for more information on hypergraphs. The algebraic translations of the invariants we will use are well-established (cf. \cite{GVV}), but are given here for the convenience of the reader. 

Throughout the paper, $\H$ will denote a hypergraph over $d$ vertices $\{x_1, \dots, x_d\}$ and $R = k[x_1, \dots, x_d]$ will be the
corresponding polynomial ring. A vertex $x \in V(\H)$ is called an
{\it isolated vertex} if $\{x\} \in E(\H)$. By definition, if $x$ is
an isolated vertex of $\H$ then $\{x\}$ is the only edge in $\H$ that
contains $x$. A {\it vertex cover} of $\H$ is a set of
vertices that has nonempty intersection with all of the edges of $\H$. We will
primarily be interested in minimal vertex covers, where minimality is with respect to inclusion. It is easy to see that there is a one-to-one correspondence between minimal vertex covers of $\H$ and minimal primes of $I(\H)$. 
The minimum cardinality of a vertex cover of $\H$ is often denoted by $\alpha_0(\H)$. Since our primary focus is square-free monomial ideals, we will denote this number by $\alpha_0(I)$, or by $\alpha_0$ when the ideal is implied. Note that
by the correspondence between minimal primes and minimal vertex covers, $\alpha_0(I)$ is also the height of $I$. 

We will refer to generators of a square-free
monomial ideal $I$ as being {\it independent} if the 
corresponding edges of the associated hypergraph are pairwise disjoint; that
is, the generators have disjoint support. 
We will denote the maximum cardinality of an independent set in $I$ by $\beta_1(I)$, or by $\beta_1$ when the ideal is implied. This agrees with the matching number $\beta_1(\H)$, which is the maximal cardinality of a matching in $\H=\H(I)$. Notice that a subset 
of the monomial generators of $I$ is independent if and only if it forms a regular sequence. Thus $\beta_1(I)$ is equal to the {\it monomial grade}, $\mgrade (I)$, of $I$, where the monomial grade of an ideal $I$ is the maximum length of a
regular sequence of monomials in $I$. Clearly, $\alpha_0(I) \ge \beta_1(I)$. 

\begin{definition}(see \cite[Chapter 2, Section 4]{Berge})
A square-free monomial ideal $I$ is said to satisfy the K\"{o}nig property if $\alpha_0(I)=\beta_1(I)$. 
Thus $I$ satisfies the K\"{o}nig property if and only if $\grade (I)= \height (I) =
\mgrade (I)$. 
\end{definition}

There are two operations commonly used on a hypergraph $\H$ to
produce a new, related, hypergraph on a smaller vertex set. Let $x \in V(\H)$ be a vertex in $\H$.
The {\it deletion} $\H \setminus x$ is formed by
removing $x$ from the vertex set and deleting any edge in $\H$ that
contains $x$. This has the effect of setting $x=0$, or of passing to
the ideal $(I(\H),x)/(x)$ in the quotient ring $R/(x)$. For convenience, we will sometimes view the deletion as its extension in the original polynomial ring. As we are primarily
concerned with relations among the generators, which are unchanged by
this extension, this will allow us to work over the original ring
without causing confusion.
The {\it contraction} $\H / x$ is obtained by 
removing $x$ from the vertex set and removing $x$ from any edge of $\H$ that contains $x$. This process has the effect of setting $x=1$, or of passing to the localization $I(\H)_x$ in $R_x$. Any hypergraph formed by a sequence of deletions and contractions is called a
{\it minor} of $\H$. The edge ideal of a minor of $\H$ is also called
a minor of $I(\H)$. Thus minors of a square-free monomial ideal can be
obtained by taking a sequence of quotients and localizations of the
original ideal. 

\begin{definition}(see \cite[Definition 1.4]{Cornuejols} and \cite[Definition 4.13]{GVV})
A square-free monomial ideal $I$ has the {\it packing} property if $I$ and
all of its minors satisfy the K\"{o}nig property. 
\end{definition}

Note that localizing at 
$P$ is equivalent to passing to a minor of $I$, and thus the packing
property is preserved under localization.


\section{Associated Primes and Unmixed Ideals} \label{sequencemethods}

In this section, we study the set of associated primes of powers of a
square-free monomial ideal $I$ when every proper minor of $I$ is normally
torsion-free. Our primary focus is to determine when $I$ is 
normally torsion-free, that is, when a power $I^t$ has no embedded
primes. We show that $I^t$ does not have any embedded primes for
$t \le \beta_1(I)$. We also show that if, in addition, $I$ is
unmixed and satisfies the K\"{o}nig property, then $I$ is normally
torsion-free.  

Recall that $R = k[x_1, \dots, x_d]$. We shall start with a simple
result (whose proof is elementary and left for the reader).

\begin{lemma}\label{exact}
Let $K$ be an ideal and let $x$ be an element in $R$. Then the
following sequence is exact: 
$$0 \rar R/(K:x) \stackrel{x}{\rar} R/K \rar R/(K,x) \rar 0.$$
\end{lemma}

The next few lemmas exhibit the behavior of associated primes of
monomial ideals when passing to subrings or larger rings obtained
by deleting or adding variables. Note that associated primes of monomial
ideals are again monomial, and are generated by subsets of the
variables.  

\begin{lemma}\label{addvariable}
Let $K$ be a monomial ideal in $R$. Let $x$ be an indeterminate of $R$
such that $x$ does not divide any minimal generator of $K$. Then there
is a one-to-one 
correspondence between the sets $\Ass (R/K)$ and $\Ass (R/(K,x))$
given by $P \in \Ass 
(R/K)$ if and only if $Q=(P, x)\in \Ass (R/(K,x))$. 
\end{lemma}

\begin{proof} Suppose $P \in \Ass (R/K)$. Then there is a monomial $c
  \in R$ such 
that $P=(K:c)$. Since $x$ does not divide any minimal generator of
$K$, we may choose a $c$ such that $x$ does not divide $c$. Clearly, $(P, x)
\subseteq ((K,x):c)$. To see the other 
inclusion, consider a monomial $f \in ((K,x):c)$. If $x | f$, then $f \in
(P,x)$. If $x$ does not divide $f$, then since $x$ does not divide
$c$, we have that $x$ does not divide $fc$. Since $fc$ is a monomial and $((K,x):c)$ is a monomial ideal, we have $fc \in K$, and so $f \in (K:c)=P$. 

Now suppose that $Q \in \Ass (R/(K,x))$. Since $x \in (K,x)
\subseteq Q$ and $Q$ is generated by a subset of the variables, we can write
$Q=(P, x)$ for some  prime ideal $P$. Let $c\in R$ be a monomial such
that $Q=(P,x)=((K,x):c)$. If $x | c$, then $((K,x):c)= R \not=P$, a 
contradiction. 
Thus, $x$ does not divide $c$. Let $y \in P$ be a minimal generator. Then $x$
does not divide $y$ and $yc \in (K,x)$. Since $(K,x)$ is a
monomial ideal, this implies that $yc \in K$. Therefore, $P \subseteq
(K:c)$. Conversely, let $g \in (K:c)$ be a minimal 
monomial generator of $(K:c)$. Since $x$ does not divide any minimal
generator of $K$, we 
have that $x$ does not divide $g$. It then follows, since $g \in (K:c)
\subseteq ((K,x):c)=Q$ and 
$x$ does not divide $g$, that $g \in P$.
\end{proof}

\begin{lemma}\label{coloninclusion}
Let $K$ be a monomial ideal and let $M$ be a monomial in $R$. Suppose $P \in \Ass (R/(K:M))$. Then $P \in \Ass (R/K)$. 
\end{lemma}

\begin{proof} Since $P\in \Ass(R/(K:M))$, there exists a monomial $c\in R$ such that
$P=((K:M):c)$. Since $((K:M):c)=(K:Mc)$, we have that $P=(K:Mc)$. Thus, $P \in \Ass (R/K)$. 
\end{proof}

The next lemma will allow us to concentrate on square-free monomial
ideals associated to connected hypergraphs. This will be useful when
passing to minors, as the minors of a hypergraph need not be
connected. The result is essentially an extension of
Lemma \ref{addvariable} and has been 
proven elsewhere for special cases (see \cite[Corollary 5.6]{SVV} for
the normally torsion-free case and see \cite[Lemma 2.1]{AJ} for the
case of the edge ideal of a graph). 

\begin{lemma}\label{disconnected}
Suppose $I$ is a square-free monomial ideal in $S = k[x_1,\dots,x_t, y_1, \dots, y_s]$ such that
$I=I_1S + I_2S$, where $I_1 \subseteq S_1 = k[x_1, \dots, x_t]$
and $I_2 \subseteq S_2 = k[y_1, \dots, y_s]$. Then $P\in \Ass (S/I^n)$ if and only if
$P=P_1S + P_2S$, where $P_1 \in \Ass (S_1/I_1^{n_1})$ and $P_2 \in \Ass (S_2/I_2^{n_2})$ with $n_1 + n_2 = n+1$.  
\end{lemma}

\begin{proof} Suppose first that $P_i \in \Ass (S_i/I_i^{n_i})$ for $i = 1,2$, and $P = P_1S + P_2S$. Then there exist monomials $c_i \in S_i$ such that $P_i = (I_i^{n_i} : c_i)$, for $i = 1,2$. Since $I_i^{n_i}$ is a monomial ideal, $P_i$ is a prime ideal generated by a subset of the variables in $S_i$. Thus, it can be seen that $c_i \in I_i^{n_i-1} \setminus I_i^{n_i}$ for $i = 1,2$. Now, if $u \in P_1$ then $uc_1c_2 \in I_1^{n_1}I_2^{n_2-1}S \subseteq I^n$. Similarly, if
$v\in P_2$ then $vc_1c_2 \in I^n$. Thus, $P \subseteq (I^n : c_1c_2)$. 
On the other hand, let $w \in S$ be a monomial such that
$wc_1c_2 \in I^n$. Since the variable sets for $S_1$ and $S_2$ are
disjoint, we have $c_1c_2 \in I^{n-1} \setminus I^n$. Write $w=w_1w_2$
where $w_1 \in S_1$ and $w_2 \in S_2$. Observe that if $w_ic_i \not\in
I_i^{n_i}$ for $i=1,2$ then $wc_1c_2 \not\in I^n$, a
contradiction. Therefore, $w_i \in P_i$ for some $i$ and so $w \in P$.  

For the converse, suppose $P \in \Ass (S/I^n)$. Again observe that $P$ is generated by a subset of the variables in $S$, and so we can write $P = P_1S + P_2S$,
where $P_1 = P \cap S_1$ and $P_2 = P \cap S_2$. Also, there exists a
monomial $c \in S$ such that $P=(I^n : c)$. As above, it can be seen
that $c \in I^{n-1} \setminus I^n$. Write $c=c_1c_2$, where $c_1 \in
S_1$ and $c_2 \in S_2$ are monomials. Then $c_1 \in 
I_1^k$ and $c_2 \in I_2^s$ for some $0 \leq k,s \leq n-1$ with $k+s =
n-1$. Suppose $x$ is a minimal generator of $P_1$. Then $x \in P = (I^n : c)$, so $xc_1c_2 \in I^n = I^{k+s+1}$. This implies that $xc_1 \in  I_1^{k+1}$. Therefore, $P_1 \subseteq (I_1^{k+1} : c_1)$. On the other hand, let $u$ be a monomial in $(I_1^{k+1} : c_1)$. Then $uc_1c_2 \in I_1^{k+1}I_2^sS = I^n$. This implies that $u \in P$. It follows that $u \in P \cap S_1 = P_1$. Therefore, $P_1 = (I_1^{k+1}:c_1)$. A similar argument shows that $P_2 = (I_2^{s+1}:c_2)$. The conclusion follows by setting $n_1 = k+1$ and $n_2 = s+1$.
\end{proof}

As observed before, associated primes behave well under localization, and so our problem can be reduced to examining when the maximal ideal $\m = (x_1, \dots, x_d)$ is an associated prime of $I^t$. 

\begin{theorem}\label{sequences}
Let $I$ be a square-free monomial ideal such that every
proper minor of $I$ is normally torsion-free. Let $y_1, \dots , y_s$ be
distinct variables in $R$, and let $\m=(x_1, \ldots , x_d)$ be the
maximal homogeneous ideal of $R$. Then $\m \in \Ass (R/I^t)$ if and
only if $\m \in \Ass (R/(I^t : \prod_{i=1}^s y_i))$. 
\end{theorem}

\begin{proof} By repeated use of Lemmas \ref{addvariable} and
  \ref{disconnected}, we may assume that the hypergraph associated to $I$ does not contain any isolated vertices. That is, we may assume that all minimal generators of $I$ are of degree at least 2.
Note also that our hypothesis implies that every minor of $I$ satisfies the packing property by \cite[Corollary 3.14]{SVV} and \cite[page 3]{Cornuejols}. 

It follows from Lemma \ref{coloninclusion} that if $\m \in \Ass (R/ (I^t : \prod_{i=1}^s y_i))$ then $\m \in \Ass (R/I^t)$. We shall use induction on $s$ to prove the other direction. By Lemma \ref{exact}, we have the following exact sequence:
$$ 0 \rar R/(I^t:y_1) \rar R/I^t \rar R/(I^t,y_1)\rar 0.$$
It then follows from \cite[Theorem 6.3]{Mat} that
\begin{align}
\Ass (R/I^t) \subseteq \Ass (R/(I^t:y_1)) \cup \Ass (R/(I^t,y_1)). \label{eq.inclusion}
\end{align}

Let $J$ be the minor of $I$ formed by deleting $y_1$. That is, the generators of $J$ are obtained from the generators of $I$ by setting $y_1=0$. By abuse of notation, we write $J^t$ for both the ideal $J^t$ in $R/(y_1)$ and its extension in $R$. Note that $J^t \subseteq I^t$, and the 
generators of $J^t$ are precisely the generators of $I^t$ that are not 
divisible by $y_1$. Thus, $(I^t,y_1)=(J^t,y_1)$.

By the hypothesis, $J$ is normally torsion-free, and so $\Ass (R/J^t)=\Min (R/J)$. It follows, since $J$ is square-free, that the maximal homogeneous ideal of $R/(y_1)$ is not an associated prime of $J^t$ unless $J$ consists of isolated vertices. Yet, isolated vertices of $J$ are also isolated vertices of $I$, and so we may assume that $J$ does not have isolated vertices. Also, by Lemma \ref{addvariable}, we have $P\in \Ass (R/(J^t ,y_1))$ if and only if $P=(P_1,y_1)$
where $P_1\in \Ass(R/J^t)=\Min(R/J)$. Therefore, if $P = (P_1,y_1) \in \Ass (R/(I^t,y_1))$ then $P_1$ is not the maximal ideal in $R/(y_1)$. That is, $\m \not\in \Ass (R/(I^t,y_1))$. It now follows from (\ref{eq.inclusion}) that if $\m \in \Ass (R/I^t)$ then $\m \in \Ass (R/ (I^t : y_1))$.

Suppose now that the assertion has been proven for a product of $s-1$
variables, and $\m \in \Ass (R/I^t)$. Let $M =
\prod_{i=1}^{s-1}y_i$. By induction, $\m \in \Ass (R/(I^t:M))$. By
Lemma \ref{exact}, we have the exact sequence 
$$ 0 \rar R/((I^t:M):y_s) \rar R/(I^t:M) \rar R/((I^t:M),y_s) \rar 0.$$
By using \cite[Theorem 6.3]{Mat} again, we have
\begin{align}
\Ass (R/(I^t:M)) \subseteq \Ass (R/((I^t:M):y_s)) \cup \Ass
(R/((I^t:M),y_s)). \label{eq.inclusion2}
\end{align}

Let $K$ be the extension in $R$ of the minor of $I$ formed by setting $y_s=0$. We shall first show that $$((I^t:M),y_s)=((K^t:M),y_s).$$
Indeed, consider a monomial $f \in ((I^t:M),y_s)$. If $y_s | f$, then $f \in ((K^t:M),y_s)$. If
$y_s$ does not divide $f$, then $f \in (I^t:M)$, and so $fM \in I^t$. Observe that $y_s$ divides neither $M$ nor $f$, so $y_s$ does not divide $fM$. Also, the generators of $K^t$ are generators of $I^t$ that are not divisible by $y_s$. Thus, $fM \in K^t$. That is, $f \in (K^t:M) \subseteq
((K^t:M),y_s)$. Conversely, consider a monomial $g \in ((K^t:M),y_s)$. If $y_s | g$,
then $g \in ((I^t:M),y_s)$. If $y_s$ does not divide $g$, then $g \in (K^t:M)$, i.e., $gM \in K^t \subseteq I^t$. Thus, $g \in (I^t:M) \subseteq ((I^t:M),y_s)$.

By Lemma \ref{addvariable}, $P \in \Ass (R/((K^t:M),y_s))$ if and
only if $P=(P_1,y_s)$ for some $P_1 \in \Ass (R/(K^t:M))$. By
Lemma \ref{coloninclusion}, $\Ass (R/(K^t:M)) \subseteq \Ass
(R/K^t)$. Also, since $K$ is a minor of $I$, our hypothesis implies
that $K$ is normally torsion-free. That is, $\Ass (R/K^t) = \Min
(R/K)$. Thus, by an argument similar to the one above, we have that
$\m \not\in \Ass (R/((K^t:M),y_s)) = \Ass (R/((I^t:M),y_s))$. This and
(\ref{eq.inclusion2}) imply that $\m \in \Ass (R/((I^t:M):y_s)) = \Ass
(R/(I^t : \prod_{i=1}^s y_i))$. The result is proved. 
\end{proof}

As a consequence of Theorem \ref{sequences}, we obtain a lower bound for the least power $t$ such that $I^t$ has embedded primes. Notice that associated primes localize, so if $P$ is an embedded prime of $I^t$ that does not contain any other embedded primes, then we can localize at $P$ and reduce to the case where $P$ is the maximal ideal.

\begin{corollary}\label{notbeforeB}
Let $I$ be a square-free monomial ideal. Assume that every proper minor of $I$ is normally torsion-free. If $\m \in \Ass (R/I^t)$ then $t \ge \beta_1(I)+1$. 
\end{corollary}

\begin{proof} For simplicity of notation, let $\beta_1 = \beta_1(I)$. By Theorem \ref{sequences}, $\m$ is associated to $I^t$ only if $\m$
is associated to $(I^t: \prod_{i=1}^d x_i)$, where the product is taken
over all distinct variables in $R$. 

Let $\{E_1, \dots, E_{\beta_1}\}$ be an independent set of generators of $I$. Then $\prod_{i=1}^d x_i$ is divisible by $\prod_{j = 1}^{\beta_1} x^{E_j} \in I^{\beta_1}$, so $\prod_{i=1}^d x_i \in I^{\beta_1}$. Thus, for $t \le \beta_1$, we have $(I^t: \prod_{i=1}^d x_i) = R$. Hence, for $t \le \beta_1$, $\m$ is not an associated prime of $(I^t : \prod_{i=1}^d x_i)$, and so $\m$ is not an associated prime of $I^t$.
\end{proof}

\begin{remark} We will see later, in Theorem \ref{thm.embeddedpower},
  that the bound in Corollary \ref{notbeforeB} is sharp when $I$ does
  not have the packing property.
\end{remark}

In the rest of this section, we will focus on unmixed ideals. Our next
result provides a better control over the colon ideal appearing in
Theorem \ref{sequences}.   

\begin{proposition}\label{powersreduce}
Let $I$ be an unmixed square-free monomial ideal satisfying the K\"{o}nig property. Let $\{E_1, \dots, E_{\beta_1}\}$ be a
maximal independent set of generators of $I$, where $\beta_1 = \beta_1(I)$, and let $g_i = x^{E_i}$ for $i = 1,\dots,
\beta_1$. If $t > \beta_1$ and $I^{t-\beta_1} =
I^{(t-\beta_1)}$, then $(I^t : \prod_{i=1}^{\beta_1} g_i) =
I^{t-\beta_1}$.  
\end{proposition}

\begin{proof} For simplicity of notation, let $M = \prod_{i=1}^{\beta_1} g_i$. It is easy to see that $(I^t : M) \supseteq I^{t - \beta_1}$. To prove the other inclusion, consider a monomial $h \in (I^t : M)$. That is, $hM \in I^t$. Then there exist $F_1, \dots, F_t \in I$ and $L \in R$ such that $hM = LF_1\cdots F_t$. 

Let $P$ be a minimal prime of $I$. Since $I$ is unmixed, we have $\height P =
\height I$. Since $I$ satisfies the K\"{onig} property, this
implies that $\height P = \beta_1$. Also, $P$ covers each of the
$g_i$'s. Thus, by the pigeonhole principle, $P$ contains precisely one
variable from each $g_i$ for $i = 1, \dots, \beta_1$. This implies that
$M \in P^{\beta_1} \setminus P^{\beta_1+1}$. Moreover, $P$ also covers
$F_i$ for $i=1, \dots, t$,  and so $hM \in P^t$. Thus, we must have $h
\in P^{t-\beta_1}$. 

Now observe that $I_P$ is a complete intersection, and that $I_P=P_P$. Thus we have $(I^r)_P=(I_P)^r=P_P^r$. This is true for any power $r$. By our hypothesis, $I^{t - \beta_1} = I^{(t - \beta_1)}$. That is, $I^{t - \beta_1}$ has no embedded primes. It follows that the primary decomposition of $I^{t-\beta_1}$ has the form
$$I^{t-\beta_1} = \bigcap_{\sqrt{Q} \in \Min(R/I)} Q.$$ 
Localizing at a minimal prime $P$, we get $P_P^{t-\beta_1}=I^{t-\beta_1}_P = Q_P$, where $Q$ is the primary ideal associated to $P$ in the above decomposition. This implies that $Q=P^{t-\beta_1}$. As a consequence, $h \in Q$. This is true for any $Q$ in the primary decomposition of $I^{t-\beta_1}$. Therefore, $h \in I^{t-\beta_1}$. Hence, $(I^t : M) \subseteq I^{t-\beta_1}$ and the result is proved.
\end{proof}

We now state a Proposition which gives an example of how our results can be applied to determine that the stable set of associated primes is only the minimal primes in the case of an unmixed ideal whose minors are normally torsion-free.

\begin{proposition} \label{thm.unmixed}
Let $I$ be an unmixed square-free monomial ideal satisfying the packing property. Assume that every proper minor of $I$ is normally torsion-free. Then $I$ is normally torsion-free. 
\end{proposition}

\begin{proof} For simplicity of notation, again let $\beta_1 =
  \beta_1(I)$. Suppose by contradiction that $I$ is
  not normally torsion-free. That is, there exists a $t$ such that $I^t$ has embedded primes. We choose $t$ minimal with respect to this property. Suppose $P$ is an embedded prime of $I^t$. Since associated primes localize and all minors of $I$ are normally torsion-free, we may assume that $P = \m$. 

By Corollary \ref{notbeforeB}, we have $t > \beta_1$. Let $\{E_1, \dots,
E_{\beta_1}\}$ be a maximal independent set of generators of $I$, and let $g_i =
x^{E_i}$. After a reindexing of the variables, we may also assume that
$x_1, \dots, x_s$ are the variables in $\prod_{i=1}^{\beta_1} g_i$. That is, $\prod_{i=1}^{\beta_1} g_i
= \prod_{i=1}^s x_i$. By Theorem \ref{sequences}, $\m \in \Ass (R/I^t)$ if and only if $\m
\in \Ass (R/(I^t:\prod_{i=1}^s x_i)) = \Ass (R/(I^t:\prod_{i=1}^{\beta_1} g_i))$. Moreover, by the choice of $t$, $I^{t-\beta_1} = I^{(t-\beta_1)}$. Thus, it follows from Proposition \ref{powersreduce} that $(I^t : \prod_{i=1}^{\beta_1} g_i) = I^{t-\beta_1}$. Now, also by the choice of $t$, $\m \not\in \Ass (R/I^{t-\beta_1 })$. Therefore, $\m \not\in \Ass (R/I^t)$, which is a contradiction. The result is proved.
\end{proof}

As a direct consequence of Proposition \ref{thm.unmixed}, we obtain the
following result. 

\begin{corollary} \label{cor.unmixed}
A minimal counterexample to the Conforti-Cornu\'{e}jols conjecture \cite[Conjecture 1.6]{Cornuejols} cannot be unmixed. 
\end{corollary}

We, in fact, can make Proposition \ref{thm.unmixed} and Corollary \ref{cor.unmixed} stronger.
By carefully examining the proof of Proposition \ref{powersreduce} and following a line of argument similar to the one used in Proposition \ref{thm.unmixed}, we can show that
if there exists a minimal generator $g$ of $I$ such
that for each minimal prime $P$ of $I$, only one generator of $P$
divides $g$ (and if all minors of $I$ are normally torsion-free) then $I$ is normally
torsion-free. Thus, if a minimal counterexample $\H$ to the Conforti-Cornu\'ejols conjecture exists, then every minimal generator of $I(\H)$ must be an
element of $P^2$ for some minimal prime $P$ of $I(\H)$.  

\begin{corollary}\label{goodedge}
Let $I$ be a square-free monomial ideal such that every minor of $I$
is normally torsion-free. If there exists a generator $g$ of $I$ such
that $g \in P\backslash P^2$ for every $P\in \Min (R/I)$, then $I$ is
normally torsion-free.
\end{corollary}

\begin{proof} Suppose that $I$ is not normally torsion-free.
As in Proposition \ref{thm.unmixed}, we choose $t$ minimal such that $I^t$ has embedded primes. 
Note that $I=I^{(1)}$ for all square-free monomial ideals, so $t \ge
2$. We claim that $(I^t:g)=I^{t-1}$. Indeed, one inclusion is
trivial, so suppose $h \in (I^t:g)$ is a monomial. Then $hg=LF_1\cdots
F_t$ for some $F_i \in I$ and $L \in R$. Thus $hg \in P^t$ for all $P
\in \Min (R/I)$. Since each such $P \in \Min(R/I)$ contains precisely one variable
that divides $g$, we must have $h \in P^{t-1}$ for all such $P$'s. Thus, as in
Proposition \ref{powersreduce}, $h \in I^{(t-1)}=I^{t-1}$, so
$(I^t:g)=I^{t-1}$. Now, as in
Proposition \ref{thm.unmixed}, we may assume $\m \in \Ass (R/I^t)$ is an embedded prime. By Theorem \ref{sequences}, $\m \in \Ass (R/I^t)$ if and only if
$\m \in \Ass(R/(I^t:g)) = \Ass (R/I^{t-1})$. Thus,
since $\m \not\in \Ass(R/I^{t-1})$, this is a contradiction and $I$ is
normally torsion-free. 
\end{proof}


\begin{example} Due to our remark above, one might hope that the packing
property would imply the existence of a minimal generator $g$ such
that $g \in P \backslash P^2$ for all minimal primes $P$ of
$I$. However, it need not be true 
for general square-free monomial ideals. For example, let $I$ be the
ideal of $k[x_1, \ldots , x_6]$ generated by $I=(x_1x_2x_3, x_4x_5x_6,
x_1x_2x_4, x_2x_3x_6, x_1x_4x_5, x_3x_5x_6)$. Then $I$ satisfies the
packing property, but $P_1 = (x_1, x_3, x_5)$ and $P_2 = (x_2, x_4,
x_6)$ are both minimal primes of $I$, and for each generator $g$ of
$I$, there is an $i\in \{ 1, 2\}$ such that $g \in P_i^2$. 
\end{example}

\begin{remark} C. Huneke and J. Mermin have informed us that they are
  also obtaining results similar to Corollary \ref{cor.unmixed} in
  their on-going research. R.H. Villarreal has shown 
  us that with some additional arguments, a non-uniform version of
  \cite[Theorem 5.10]{GRV} can be used to recover Proposition
  \ref{thm.unmixed}. 
\end{remark} 

We conclude this section by giving an example of how the preceding
results can be used to determine the associated primes of powers of a
given square-free monomial ideal. 

\begin{example}
Suppose $\H$ is a degree three five-cycle connected in codimension
one. Then
$I=I(\H)=(x_1x_2x_3,x_2x_3x_4,x_3x_4x_5,x_4x_5x_1,x_5x_1x_2)$.
Any minor of $I$ formed by a contraction, or inverting one of the
variables, will be a bipartite graph, and thus by \cite{SVV} such a
minor is normally torsion-free. Since any minor of a bipartite graph
is again a bipartite graph, and the order in which a mixed minor is
formed commutes, any minor of $I$ formed by using at least one
localization will be normally torsion-free. If a minor is formed using
only deletions (quotients), then if two or more variables are deleted,
the minor is either principal (if the deleted vertices are of the form
$x_ix_{i+1}$ modulo $5$) or vanishes. So without loss of
generality, form a minor by deleting $x_1$. The resulting ideal is
$(x_2x_3x_4,x_3x_4x_5)$, which satisfies the hypotheses of Corollary
\ref{goodedge}, with either generator serving as the generator
$g$. Thus this minor is also normally torsion-free. Hence the only
possible embedded associated prime for any power of $I$ is $\m$. One
can easily verify that $\m \in \Ass(R/I^t)$ for all $t\geq 2$ by
noting that if $c= x_1x_2(x_3x_4x_5)^{t-2}$, then $\m = (I^t:c)$.
Thus the stable set of associated primes is $\Min(R/I) \cup \{\m\}$ and
this set is reached for $\Ass(R/I^2)$. 
\end{example}


\section{Polarization and Embedded Associated Primes} \label{polarization}

In this section, we focus on square-free monomial ideals that are
minimally not normally torsion-free. We show that if in addition the
packing property fails to hold then embedded primes appear at the
$(\beta_1+1)$-st power of the ideal, where 
$\beta_1=\beta_1(I)$. This
further shows that the bound given in Corollary \ref{notbeforeB} is sharp. 

Throughout the section, $I \subseteq R = k[x_1, \dots, x_d]$ will
denote a square-free monomial ideal. Our method in this section is to
use {\it polarization} (see, for example, \cite{MS} for a more detailed
discussion about polarization). 

\begin{definition} The process of {\it polarization} replaces a
power $x_i^t$ by a product of $t$ new variables $x_{(i,1)} \cdots x_{(i,t)}$. 
We call $x_{(i,j)}$ a {\it shadow} of $x_i$. 
We will use $\widetilde{I^t}$ to denote the polarization of
$I^t$, will use $S_t$ for the new polynomial ring in this polarization, and will use $\widetilde{w}$ to denote the polarization in $S_t$ of a monomial $w$ in $R$. The {\it depolarization} of an ideal in $S_t$ is formed by setting $x_{(i,j)}=x_i$ for all $i,j$. 
\end{definition}

Observe that the polarization of a power $I^t$ of $I$ is a square-free monomial ideal in $d \cdot t$ variables. Note that if $x_{(i,j)}$ divides a minimal generator $M$ of $\widetilde{I^t}$, then
$x_{(i,k)}$ divides $M$ for all $1 \le k \le j$. Note also that the depolarization of
$\widetilde{I^t}$ is $I^t$.  

We begin with a lemma showing that a minimal prime of
$\widetilde{I^t}$ cannot contain multiple variables that are
shadows of the same variable in $R$. This will restrict the class of
primes to be considered when dealing with polarizations.  

\begin{lemma} \label{oneshadow}
Let $I$ be a square-free monomial ideal and let $P$ be a minimal prime
of $\widetilde{I^t}$ in $S_t$ for some $t$, and suppose 
$x_{(i,j)} \in P$. Then $x_{(i,k)} \not\in P$ for all $k \not= j$.  
\end{lemma}

\begin{proof} Let $\H_t$ be the hypergraph associated to
  $\widetilde{I^t}$. Then $P$ is a minimal vertex cover of
  $\H_t$. Suppose by contradiction that $x_{(i,j)}$ and $x_{(i,k)}$
  are both in $P$ and $k \not= j$. Without loss of generality, assume
  $k < j$. Let $v$ be a minimal generator of $\widetilde{I^t}$ that is
  covered by $x_{(i,j)}$. From our observation above, $v$ is divisible
  by $x_{(i,l)}$ for all $l \le j$. In particular, $v$ is divisible by
  $x_{(i,k)}$. Thus, $P \setminus \{x_{(i,j)} \}$ is a vertex cover of
  $\H_t$. This is a contradiction to the minimality of $P$. The lemma
  is proved. 
\end{proof}

\begin{remark} \label{coverlift}
Observe that every minimal prime of $I$ lifts to a minimal prime of the
polarization $\widetilde{I^t}$ of $I^t$ for every $t$. Indeed, if
$(x_1, \dots, x_r)$ is a minimal prime of $I$, then $(x_1, \dots,
x_r)$ is a minimal prime of $I^t$. This implies that $\{ x_{(1,1)},
\ldots , x_{(r,1)}\} $ is a vertex cover for the hypergraph associated to $\widetilde{I^t}$. This cover is necessarily minimal. In other words, $(x_{(1,1)}, \dots, x_{(r,1)})$ is a minimal prime of $\widetilde{I^t}$. 
\end{remark}



In the next lemma, we rephrase a result of Faridi \cite{Fa1} which shows that the embedded primes of $I^t$ also lift to associated primes of $\widetilde{I^t}$, and that associated primes of
the $\widetilde{I^t}$ depolarize to associated primes of $I^t$. This
creates a correspondence, which is not usually one-to-one, between
associated primes of $I^t$ and associated primes of its polarization
$\widetilde{I^t}$. Notice that $\widetilde{I^t}$ is a square-free
monomial ideal, and so all associated primes are minimal. Thus
statement $(1)$ below is actually a statement about all associated
primes of the polarization.

\begin{lemma} \label{primesPolarize} Let $I$ be a square-free monomial
  ideal, and let $t$ be a positive integer. 
\begin{enumerate}
\item Let $P\in \Min (R/\widetilde{I^t})$, and let $p$ be the
  depolarization of $P$. 
Then $p \in \Ass (R/I^t)$. 
\item Let $q \in \Ass(R/I^t)$. Then, there is at least one prime $Q
  \in \Ass (R/\widetilde{I^t})$ 
such that the depolarization of $Q$ is $q$.
\end{enumerate}
\end{lemma}

\begin{proof} See Corollary 2.6 of \cite{Fa1}.
\end{proof}

\begin{example} Notice that the correspondence between $q \in \Ass(R/I^t)$ and $Q \in \Ass(R/\widetilde{I^t})$ in Lemma \ref{primesPolarize} is not one-to-one. For example, let $R=k[x,y,z]$ and let $I=(xy,yz,xz)$ be the edge
ideal of a triangle. Then $I^2=(x^2y^2, y^2z^2,x^2z^2, xy^2z, xyz^2,
x^2yz)$ and 
$$\widetilde{I^2}=(x_1x_2y_1y_2, y_1y_2z_1z_2, x_1x_2z_1z_2,
x_1y_1y_2z_1, x_1y_1z_1z_2, x_1x_2y_1z_1).$$
The associated primes of
$R/I^2$ are $\{ x,y\}, \{x,z\}, \{y,z\}, \{x,y,z\}$ and the associated
primes of $S/\widetilde{I^2}$ are 
$$\begin{array}{ccccl} \{x_1,y_1\}, & \{x_2,y_1\}, & \{x_1, y_2\},
& \{x_1,z_1\}, & \{x_2, z_1\}, \\ \{y_1,z_1\}, & \{y_2,z_1\}, & \{x_1,z_2\},
& \{y_1,z_2\}, & \{x_2, y_2, z_2\}. \end{array}$$
\end{example}

We are now ready to prove our main result in this
section. Notice that the ideals in this theorem are minimally non-packing.

\begin{theorem} \label{thm.embeddedpower}
Let $I \subseteq R = k[x_1, \dots, x_d]$ be a square-free monomial ideal such that every minor of $I$ is normally torsion-free. Suppose that $I$ does not have the packing property. Then $\m \in \Ass (R/I^{\beta_1(I) +1})$ and $\beta_1(I)+1$ is the least power $t$ such that $\m \in \Ass (R/I^t)$.
\end{theorem}

\begin{proof} The second statement follows from Corollary \ref{notbeforeB}. We shall prove the first statement. For simplicity of notation, let $\beta_1 = \beta_1(I)$. By definition, it can be seen that $\prod_{i=1}^d x_i \in I^{\beta_1} \setminus I^{\beta_1 +1}$. Thus no minimal generator of $I^{\beta_1 +1}$ is square-free. This implies
that $A = \{x_{(1,2)}, x_{(2,2)}, \dots, x_{(d,2)}\}$ is a vertex cover of the
associated hypergraph $\H'$ of $\widetilde{I^{\beta_1 +1}}$. 

We claim that $A$ is a minimal vertex cover of $\H'$. Suppose by contradiction that $A$ is not. Then there is a subset $B$ of $A$ that is a minimal vertex cover of $\H'$. Let $Q$ be the prime ideal generated by elements in $B$. By Lemma \ref{primesPolarize}, $Q$ depolarizes to an associated prime $q$ of $I^t$. Since every minor of $I$ is normally torsion-free, a
localization argument shows that the only possible embedded prime of $I^t$ is
the maximal homogeneous ideal $\m$ of $R$. This implies that $Q$ depolarizes to a minimal
prime of $I^t$. That is, $q$ is a minimal prime of $I^t$, and so $q$ is a minimal prime of $I$. By reindexing the variables in $R$ if necessary, we may assume that $q=(x_1,\dots, x_s)$. Then $C = \{x_1, \dots, x_s\}$ is a minimal vertex cover of the hypergraph $\H$ associated to $I$. By definition, for each $x_i \in C$, there exists a monomial generator $g_i$ of $I$ such that $g_i$ is not covered by $C \setminus \{x_i\}$. 
It follows from our hypothesis that every minor of $I$ has the packing property by \cite[Corollary 3.14]{GVV} and \cite[page 3]{Cornuejols}. Thus, our hypothesis implies that $I$ does not have the K\"onig property. That is, $\alpha_0(I) > \beta_1(I)$. As observed before, $\alpha_0(I) = \height I$, so $s \ge \alpha_0(I) \ge \beta_1 + 1$. Observe now that for any $j = 1, \dots, s$, $x_j^2$ does not divide $M=\prod_{i=1}^{\beta_1+1} g_i$. This implies that 
the polarization $\widetilde{M}$ in $\widetilde{I^{\beta_1+1}}$ is not covered
by $(x_{(1,2)}, \dots , x_{(s,2)})$, a contradiction to the fact that $B$ is a vertex cover of $\H'$. 

We have shown that $A$ is a minimal vertex cover of the hypergraph $\H'$ associated to $\widetilde{I^{\beta_1+1}}$. Let $P$ be the ideal generated by elements in $A$. Then $P$ is a minimal prime of $\widetilde{I^{\beta_1+1}}$. It then follows from Lemma \ref{primesPolarize} that $\m$, which is the depolarization of $P$, is an associated prime of $I^{\beta_1+1}$. The result is proved. \end{proof}

\begin{remark} Theorem \ref{thm.embeddedpower} shows that the bound in Corollary \ref{notbeforeB} is sharp.
\end{remark}

Notice that the above remark, if combined with a proof of \cite[Conjecture 1.6]{Cornuejols}, would provide a practical bound for how many
powers one needs to check to guarantee that any particular square-free
monomial ideal is normally torsion-free. For any ideal $I$ of $R$, the set of
associated primes of $R/I^n$ stabilize for large $n$ (cf. \cite{brod}),
and so there exists a finite integer $N$ such that if
$I^n=I^{(n)}$ for all $n \leq N$, then $I^n=I^{(n)}$ for all $n$. 
For classes of non-monomial ideals, the minimal such $N$ has been
studied and is often of the order of the dimension $d$ of $R$ (cf.
\cite{morey}). For square-free monomial ideals, Theorem
\ref{thm.embeddedpower} indicates that the bound on $N$ should
actually be $\beta_* +1$ where $\beta_* =  {\mbox {\rm max}} \{ \beta_1(I')
\, | \, I' \, {\mbox {\rm is a minor of }} I\}$. 

\begin{remark} \label{rmk.highestpower}
Suppose that \cite[Conjecture 1.6]{Cornuejols} is true. If $I$ is any square-free
monomial ideal and $I^n=I^{(n)}$ for all $n\leq
N=\lceil {\frac{d+1}{2}}\rceil$, then $I^n=I^{(n)}$ for all $n$.
\end{remark}

\begin{proof}
Suppose $I^t$ has an embedded prime for some $t$. Let $P$ be an
embedded prime of $I^t$ such that $P$ does not contain any other
embedded primes. Then $PR_P$ is the maximal ideal of $R_P$ and is the
unique embedded prime of $(I_P)^t$. Let $\H$ be the hypergraph associated
to $I_P$. By Lemma \ref{disconnected}, we can assume that $\H$ does not have
isolated vertices. Since $I_P$ is not normally torsion-free and we
have assumed \cite[Conjecture 1.6]{Cornuejols} holds, we can assume $I$ does not have the
packing property. If $I$ is minimally non-packing, by Theorem
\ref{thm.embeddedpower}, $PR_P$ is associated to $(I_P)^{\beta_1(I)+1}$
and thus $P$ is associated to $I^{\beta_1(I)+1}$. Since each edge of
$\H$ contains at least two vertices and the number of vertices of $\H$
is at most $d$, we have $\beta_1(I)+1 \leq N$. If
$I$ is not minimally non-packing, we can pass to a minor $I'$ of
$I$ that is minimally non-packing, and again assume that $I'$ does
not have isolated vertices. By Theorem \ref{thm.embeddedpower}, the power
$(I')^{\beta_1(I')+1}$ has an embedded prime. As before,
$\beta_1(I')+1 \leq N$, and by Lemma \ref{addvariable} and the fact
that associated primes localize, we again get that $I^{\beta_1(I')+1}$
also has an embedded prime. Thus in any case, if $I^t$ has an embedded
prime for some $t$, then $I^n$ has an embedded prime for some $n\leq
N$, as desired.
\end{proof} 

The bound in Remark \ref{rmk.highestpower} is sharp. For example, edge ideals of odd cycles achieve this bound (see \cite{AJ}). We conjecture that the bound in Remark \ref{rmk.highestpower} is true regardless of the validity of \cite[Conjecture 1.6]{Cornuejols}. This would allow one to check the MFMC property, and therefore, the packing property of a hypergraph $\H$, by computing the associated primes of a uniformly bounded number of powers of $I(\H)$.

\end{document}